\documentclass{article}
\usepackage{graphicx} 
\usepackage[utf8]{inputenc}
\usepackage[linktocpage=true,hidelinks]{hyperref}
\usepackage{setspace}
\usepackage{amssymb, amsmath, mathtools, amsthm, graphicx,mathrsfs}
\usepackage{caption}
\usepackage[noadjust]{cite}
\usepackage{mathtools,xcolor, tikz}
\usepackage{titling}
\usepackage{enumerate}
\usepackage{cleveref}
\usetikzlibrary {arrows.meta,graphs}

\newtheorem{thm}{Theorem}[section]

\newtheorem{lem}[thm]{Lemma}
\newtheorem{conj}[thm]{Conjecture}
\newtheorem{prop}[thm]{Proposition}

\newtheorem{defi}{Definition}

\newcommand{\C}{\mathcal{C}}

\renewcommand{\l}{\left}
\renewcommand{\r}{\right}

\title{Decomposition of Cliques into $k$-Star-Forests}
\author{Jiaxi Nie\thanks{School of Mathematics, Georgia Institute of Technology, Atlanta, GA 30332, USA. Email: {\tt jnie47@gatech.edu}.}
 \and Yibo Ren\thanks{School of Mathematics, Fudan University, Shanghai 200438, China. Enail: {\tt ybren24@m.fudan.edu.cn}.} \and Hehui Wu\thanks{Shanghai Center for Mathematical Sciences, Fudan University, Shanghai 200438, China. Research supported by National Natural Science Foundation of China grant 12371343 and grant 11931006, National Key Research and Development Program of China 2020YFA0713200. Email: {\tt hhwu@fudan.edu.cn.}}}
\date{\today}

\begin{document}

\maketitle

\begin{abstract}
A $k$-star-forest is a forest with at most $k$ connected components where each component is a star. Let $F_k(n)$ be the minimum integer such that the complete graph on $n$ vertices can be decomposed into $F_k(n)$ $k$-star-forests. Pach, Saghafian and Schnider showed that $F_2(n)=\lceil 3n/4 \rceil$. In this paper, we show that $F_3(n)=5n/9$ when $n$ is a multiple of 27. Further, for $k\ge 4$, we show that $F_k(n)=n/2+2$ when $n>2k$ and $n\equiv 4 \pmod{12}$. Our results disprove a conjecture of Pach, Saghafian and Schnider.
\end{abstract}

\section{Introduction}
It is a classic theme in combinatorics to determine the minimum number of subgraphs of a certain type that a graph $G$ can be decomposed into. Here a decomposition of a graph $G$ is a family of subgraphs that partitions $E(G)$. For example, the seminal result of Graham and Pollak\cite{graham1971addressing} shows that $K_n$, the complete graph on $n$ vertices, cannot be decomposed into less than $n-1$ complete bipartite graphs (on the other hand, it is easy to see that $K_n$ can be decomposed into $n-1$ stars). Other known results under this theme have also been obtained by Vizing~\cite{vizing1964estimate} for matchings, by Lovász~\cite{lovasz1968covering} for paths and cycles, and by Nash-Williams~\cite{nash1964decomposition} for forests. 

 A {\em star-forest} is a forest whose components are all stars. Akiyama and Kano~\cite{akiyama1984path} proved that $K_n$ cannot be decomposed into less than $\lceil n/2\rceil+1$ star-forests and this bound is tight.

 For integer $k\ge 1$, a {\em $k$-star-forest} is a star-forest with at most $k$ connected components. Recently, Pach, Saghafian and Schnider~\cite{pach2023decomposition} studied the decomposition of cliques into $k$-star-forests. Let $F_k(n)$ be the minimum integer such that the complete graph on $n$ vertices can be decomposed into $F_k(n)$ $k$-star-forests. Pach, Saghafian and Schinider proved the following theorem.

\begin{thm}[\cite{pach2023decomposition}, Theorem 4] \label{thm:2-forest}
For $n\ge 3$, $F_2(n)=\lceil3n/4\rceil$.
\end{thm}

We use the notation $S(v;u_1,\dots,u_t)$ to denote the star with center $v$ and leaves $u_1,\dots,u_t$. The construction for \Cref{thm:2-forest} can be described as follows. For simplicity, we will only describe it for even integer $n$. Let $n=2t$ and label the vertices as $\{v_1,\dots,v_{2t}\}$ (indices are modulo $2t$). For $1\le i\le t$, let $S_i$ be the 2-star-forest consisting of stars $S(v_i;v_{i+1},\dots,v_{i+t-1})$ and $S(v_{i+t};v_{i+t+1},\dots,v_{i+2t-1})$. One can check that $S_1,\dots,S_t$ cover all edges but those of the form $v_iv_{i+t}$. This construction, called {\em broken double star} in~\cite{antic2024star}, shows that $K_n$ can be decomposed into at most $n/2+1$ star-forests. Further, Anti{\'c}, Gli{\v{s}}i{\'c} and Milivoj{\v{c}}evi{\'c}~\cite{antic2024star} showed that such decomposition is unique up to isomorphism. This result will be useful later in this paper.

\begin{thm}[\cite{antic2024star}, Theorem 1]\label{thm:broken double star}
Let $n = 2t$ be an even integer. Then any decomposition of $K_n$ into
$t+1$ star-forests is a broken double star decomposition.
\end{thm}

Note that for 2-star-forests,  one can decompose the edges of the form $v_iv_{i+t}$ into $\lceil n/4 \rceil$ matchings of size at most 2, which provides the construction for \Cref{thm:2-forest}. Similarly, if the goal is to decompose $K_n$ into $k$-star-forests, then one can make use of $S_1,\dots, S_t$ plus $\lceil \frac{n}{2k} \rceil$ matchings of size at most $k$. Pach, Saghafian and Schider conjectured that this construction is best possible.
\begin{conj}[\cite{pach2023decomposition}]\label{conj:star}
For any $n\ge k\ge 2$, $F_k(n)\ge \l\lceil\frac{(k+1)n}{2k} \r\rceil$.
\end{conj}

In this paper, we determine the value of $F_k(n)$ for every $k\ge 3$ and infinitely many integers $n$, which disproves \Cref{conj:star}.

\begin{thm}\label{thm:3-star}
$F_3(n)=5n/9$ when $n$ is positive and a multiple of 27.
\end{thm}

\begin{thm}\label{thm:4-star}
For $k\ge 4$, $F_k(n)=\frac{n}{2}+2$ when $n>2k$ and $n\equiv 4 \pmod{12}$.
\end{thm}

In fact, by \Cref{thm:broken double star} we know that $F_k(n)\ge \frac{n}{2}+2$ when $n> 2k$, and it turns out that this lower bound is tight for $k=4$ when $n=12m+4 (m\ge 1)$. Recall the definition of $k$-star-forest, this actually proves \Cref{thm:4-star} for $k\ge4$.

The rest of this paper is structured as follows. Section~\ref{sec:3-star} consists of three subsections where in Subsection~\ref{subsec:3-star-lowerbound} we prove the lower bound for \Cref{thm:3-star} and in Subsections~\ref{subsec:n=27} and ~\ref{subsec:blowup} we present constructions that match the lower bound. In Section~\ref{sec:4-star}, we prove Theorem~\ref{thm:4-star} for $k=4$ by first presenting a construction for $n=16$ in Subsection~\ref{subsec:16}, and then extend it to $n=12m+4$ in the rest of Sections~\ref{sec:4-star}.

\section{Decomposition into 3-star forests}\label{sec:3-star}
In this section, we prove Theorem~\ref{thm:3-star}.

\begin{defi}
Let $\mathcal{C}$ be a collection of star-forests. The \textbf{root-hypergraph} $B$ of $\mathcal{C}$ is defined as follows: every star-forest $S$ in the decomposition corresponds to a unique hyperedge $e_S$ in $B$, such that vertices in $e_S$ are exactly centers of the stars in $S$.  
\end{defi}
For example, the root-hypergraph of the broken double star described earlier has edges $\{v_i,v_{i+t}\}$ for each $1\le i\le t$ and $\{v_1,~v_2,~\dots,~v_t\}$. This concept is useful in our following proofs, and was used by Pach, etc. in~\cite{pach2023decomposition} to prove their result for 2-star-forests (Theorem 4), although they did not name it. 

\begin{prop}\label{prop:no-iso}
Let $\C$ be a collection of star-forests decomposing $K_n$ and let $B$ be its root-hypergraph. If $|\C|<n-1$, then $B$ has no isolated vertex.
\end{prop}

\begin{proof}
Suppose there exists a vertex $v$ not contained in any hyperedges in $B$. By definition it means that $v$ is not a center for any star-forests in $\C$, which implies that each star-forest contain at most one edge containing $v$. Thus, to cover all $(n-1)$ edges containing $v$, there are at least $n-1$ star-forests in $\C$. 
\end{proof}

	
\subsection{Lower bounds}\label{subsec:3-star-lowerbound}
In this subsection, we show that $F_3(n)\ge \frac{5n}{9}$ for all $n\ge 3$. Suppose (for contradiction that) there exists a minimal $n\ge3$ such that $F_3(n) < \frac{5n}{9}$, giving rise to a decomposition of $K_n$ by $F_3(n)$ 3-star-forests.  Let $\C_0$ be the corresponding collection of 3-star-forests and let $B_0$ be its root-hypergraph. Note that $|\C_0|<5n/9<n-1$, so by \Cref{prop:no-iso} $B_0$ has no isolated vertex. Moreover, one can show that every hyperedge in $B_0$ contains at least two vertices. Indeed, if $\C_0$ contains a 1-star-forest $S$, then by deleting the center of $S$ from $K_n$ as well as $S$ itself from $\C_0$, we have $F_3(n-1)\le F_3(n)-1 < \frac{5(n-1)}{9}$, contradicting with the minimality of $n$.
	
Let $m=F_3(n)$ and $r$ be the number of hyperedges in $B_0$ of size 2. In other words, $r$ is the number of $2$-star-forests in $\C_0$. Thus there are $(m-r)$ hyperedges in $B_0$ of size 3, and we have
$$
\sum_{P\in V(B_0)}deg_{B_0}(P)=\sum_{e\in B_0}|e|=2r+3(m-r)=3m-r,
$$
where $deg_{B_0}(P)$ denotes the degree of $P$ in $B_0$, i.e. the number of hyperedges in $B_0$ containing $P$.
	
Since $B_0$ has no isolated vertices, every vertex has degree $\ge1$. Let us focus on vertices with degree 1 in $B_0$, and denote the set of these vertices by $V_1$, the set of other vertices by $V_{\ge 2}$.
\begin{lem}\label{Lem:3}
\begin{itemize}
    \item[(1)] No hyperedge in $B_0$ contain two vertices in $V_1$.
    \item[(2)] For every vertex of degree 2 in $B_0$, the two hyperedges containing it will not both contain some vertex in $V_1$.
\end{itemize}
\end{lem}
\begin{proof}
\begin{itemize}
    \item [(1)] If there exists a hyperedge $e$ in $B_0$ containing $P, Q\in V_1$, then the star-forest corresponding to $e$ doesn't cover the edge $PQ$, since $P$ and $Q$ are both centers in that star-forest. Further, since $P,Q\in V_1$, no other star-forests use $P$ or $Q$ as center, thus $PQ$ is also not covered by any other star-forests. This contradicts with $\C_0$ decomposing $K_n$.
    \item [(2)] Suppose that a degree 2 vertex $P$ in $ B_0$ is connected to $Q_1, Q_2\in V_1$ by 2 hyperedges in $B_0$ respectively, and that $Q_1$ is center of the star containing $Q_1Q_2$ in the decomposition, without loss of generality. Then the star-forest corresponding to the hyperedge containing $P$ and $Q_1$ cannot contain edge $PQ_2$ in $K_n$. Further, the star-forest corresponding to the hyperedge connecting $P$ and $Q_2$ cannot contain edge $PQ_2$ either. Thus $PQ_2$ is not covered by any star-forests, a contradiction.
\end{itemize}
\end{proof}
	
Let $p_j$ be the number of degree-$j$ vertices in $B_0$. Then $|V_1|=p_1$. Calculate the number and degree sum of vertices we have:
\begin{equation}\label{eq:3-1}
        \sum_{j\geqslant1}p_j=n;
\end{equation}
\begin{equation}\label{eq:3-2}
		\sum_{j\geqslant1}j\cdot p_j=3m-r.
\end{equation}
	
Furthermore, consider the bipartite graph $B_1$ between $V_1$ and $V_{\ge 2}$ such that for all $P\in V_1$ and $Q\in V_{\ge2}$, edge $PQ$ is in $B_1$ if and only if $P, Q$ is connected by a hyperedge in $B_0$ (see \Cref{figure:B1} for an example). By \Cref{Lem:3}(1), every vertex $P\in V_1$ has degree 1 or 2 in $B_1$, which is determined by the size of the hyperedge connecting $P$ in $B_0$. Thus the number of edges in $B_1$ is at least $r+2(p_1-r)=2p_1-r$. Note that the degree of every vertex $Q\in V_{\ge 2}$ in $B_1$ is controlled by its degree in $B_0$: $deg_{B_1}(Q)\le j$ when $deg_{B_0}(Q)=j\ge3$, and $deg_{B_1}(Q)\le 1$ when $deg_{B_0}(Q)=2$ by \Cref{Lem:3}(2). Thus by double counting the number of edges in $B_1$ we have
$$
2p_1-r\leqslant p_2+\sum_{j\geqslant3}j\cdot p_j.
$$
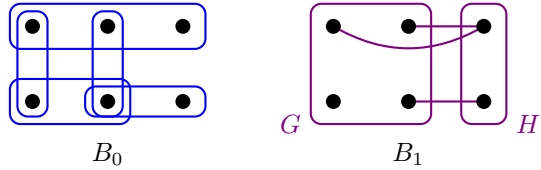
\begin{figure}[h]
		\centering
		\begin{tikzpicture}
			\node foreach \x in {1,2,3}foreach \y in{1,2}[circle,fill,inner sep=2] at (\x,\y){};
			\foreach \x in {1,2}{
				\draw[blue,thick,rounded corners] (\x-0.2,0.8)rectangle(\x+0.2,2.2);
				\draw[blue,thick,rounded corners] (\x-0.3,0.6+0.1*\x)rectangle(\x+1.3,1.4-0.1*\x);
			}
			\draw[blue,thick,rounded corners](0.7,1.7)rectangle(3.3,2.3);
			\node at(2,0.3) {$B_0$};
			
			\draw[violet,thick](7,1)edge(6,1) (7,2)edge(6,2) (7,2)edge[bend left](5,2);
			\node foreach \x in {1,2,3}foreach \y in{1,2}[circle,fill,inner sep=2] at (\x+4,\y){};
			\draw[violet,thick,rounded corners] (6.7,2.3)rectangle(7.3,0.7)node[right]{$H$};
			\draw[violet,thick,rounded corners] (6.3,2.3)rectangle(4.7,0.7)node[left]{$G$};
			
			\node at(6,0.3) {$B_1$};
		\end{tikzpicture}
		\caption{An example of $B_0$ and its corresponding $B_1$.}
            \label{figure:B1}
\end{figure}
We add $3$ times \Cref{eq:3-2} and minus $5$ times \Cref{eq:3-1} to the inequality above, resulting in
$$
5n-9m+2r\leqslant-\sum_{j\geqslant3}(2j-5)p_j\leqslant0,
$$
contradicts with $m<\frac{5n}{9}$.
	
\subsection{The construction for $n=27$}\label{subsec:n=27}
From the inequality above, the equality holds only when $r=p_j(j\geqslant3)=0$, that means no 2-star-forest are used and each vertex acts as a center in some star-forest at most twice. If $F_3(n)=5n/9$, then $n$ must be a multiple of $9$. For $n=9,18$, it is easy to check $(n-3)\cdot \frac{5n}{9}<\binom{n}{2}$, hence it is impossible to decompose $K_n$ into $\frac{5n}{9}$ 3-star-forests. For $n=27$, we discover the following construction (see \Cref{figure:27} for an illustration).
	
Let $\mathbb{F}^3_3=\{(i,j,k):i,j,k\in\mathbb{F}_3\}$ be the vertex set of $K_{27}$ where $\mathbb{F}_3=\{0,1,2\}$ is the finite field of size 3. Let $\{(i,j,0),(i,j,1),(i,j,2)\}$ be the centers of the 3-star-forest $S_{ij}$ consisting of stars
\begin{align*}
		S ((i,j,0); &(i,j+1,0),(i+1,j,0),(i+1,j-1,0),(i-1,j-1,0),\\
		&(i,j-1,1),(i-1,j,1),(i-1,j+1,1),(i+1,j+1,1),\\
            &(i,j-1,2), (i-1,j,2), (i-1,j+1,2), (i+1,j+1,2) )\\
S ((i,j,1); &(i+1,j,1), (i-1,j-1,1),\\
&(i+1,j,2), (i-1,j-1,2),\\
&(i-1,j,0), (i+1,j+1,0) )\\
S ((i,j,2); &(i,j+1,2), (i+1,j-1,2),\\
&(i,j+1,1), (i+1,j-1,1),\\
&(i,j-1,0), (i-1,j+1,0)).
\end{align*}

Let $X_j=\{(0,j,1), (1,j,1), (2,j,1)\}$ be the centers of the 3-star-forest $S_{X_j}$ consisting of stars
\begin{align*}
	S ((i,j,1);~ &(i+1,j-1,1), (i,j+1,1),\\
	&(i+1,j-1,2), (i,j+1,2), (i,j,2),\\
	&(i-1,j+1,0), (i,j-1,0), (i,j,0) ),~~~\forall~ 0\le i\le 2.
\end{align*}

Let $Y_i=\{(i,0,2),(i,1,2),(i,2,2)\}$ be the centers of the 3-star-forest $S_{Y_i}$ consisting of stars
\begin{align*}
		S ((i,j,2); &(i+1,j,2),(i-1,j-1,2),\\
		&(i+1,j,1),(i-1,j-1,1),\\
		&(i-1,j,0),(i+1,j+1,0),(i,j,0) ),~~~\forall 0\le j\le 2.
\end{align*}

\begin{figure}[h]
		\centering
		\begin{tikzpicture}[scale=.8]
			\draw[gray,-latex](1,0,1)--(3.3,0,1)node[above]{$j$};
			\draw[gray,-latex](1,0,1)--(1,2.3,1)node[left]{$k$};
			\draw[gray,-latex](1,0,1)--(1,0,3.6)node[left]{$i$};
			\foreach \x in {1,2,3}\foreach \y in {1,2,3}{
				\fill (\x,0,\y)circle[radius=2pt];
				\fill (\x,1,\y)circle[radius=2pt];
				\fill (\x,2,\y)circle[radius=2pt];
				\draw[blue,thick] (\x,0,\y)--(\x,2,\y);
				\draw[blue,thick] (\x,1,1)--(\x,1,3);
				\draw[blue,thick] (1,2,\y)--(3,2,\y);}
			\node[align=center] at(2,-1,2) {root-hypergraph};
		\end{tikzpicture}
		\begin{tikzpicture}[scale=.9]
			\foreach \x in {1,2,3}\foreach \y in {1,2,3}{
				\fill (\x,0,\y)circle[radius=1.5pt];
				\fill (\x,1,\y)circle[radius=1.5pt];
				\fill (\x,2,\y)circle[radius=1.5pt];}
			\begin{scope}[every path/.style={blue,arrows={>-Stealth[fill=white]}}]\path
				(2,0,2)edge(3,0,2)	(2,0,2)edge(2,0,3)	(2,0,2)edge(1,0,3)	(2,0,2)edge(1,0,1)
				(2,0,2)edge(1,1,2)	(2,0,2)edge(2,1,1)	(2,0,2)edge(3,1,1)	(2,0,2)edge(3,1,3)
				(2,0,2)edge(1,2,2)	(2,0,2)edge(2,2,1)	(2,0,2)edge(3,2,1)	(2,0,2)edge(3,2,3);
			\end{scope}
			\begin{scope}[every path/.style={purple,arrows={>-Stealth[fill=white]}}]\path
				(2,1,2)edge(2,1,3)	(2,1,2)edge(1,1,1)	(2,1,2)edge(2,2,3)	(2,1,2)edge(1,2,1)	(2,1,2)edge(2,0,1)	(2,1,2)edge(3,0,3);
			\end{scope}
			\begin{scope}[every path/.style={olive,arrows={>-Stealth[fill=white]}}]\path
				(2,2,2)edge(3,2,2)	(2,2,2)edge(1,2,3)	(2,2,2)edge(3,1,2)	(2,2,2)edge(1,1,3)	(2,2,2)edge(1,0,2)	(2,2,2)edge(3,0,1);
			\end{scope}
			\node[align=center] at(2,-1,2) {(I) $S_{{ij}}$};
		\end{tikzpicture}
		\begin{tikzpicture}[scale=.9]
			\foreach \x in {1,2,3}\foreach \y in {1,2,3}{
				\fill (\x,0,\y)circle[radius=1.5pt];
				\fill (\x,1,\y)circle[radius=1.5pt];
				\fill (\x,2,\y)circle[radius=1.5pt];}
			\begin{scope}[every path/.style={blue,arrows={>-Stealth[fill=white]}}]\path
				(2,1,1)edge(1,1,2)	(2,1,1)edge(3,1,1)	(2,1,1)edge(1,2,2)	(2,1,1)edge(3,2,1)
				(2,1,1)edge(3,0,3)	(2,1,1)edge(1,0,1)	(2,1,1)edge(2,0,1)	(2,1,1)edge(2,2,1);
			\end{scope}
			\begin{scope}[every path/.style={purple,arrows={>-Stealth[fill=white]}}]\path
				(2,1,2)edge(1,1,3)	(2,1,2)edge(3,1,2)	(2,1,2)edge(1,2,3)  (2,1,2)edge(3,2,2)
				(2,1,2)edge(3,0,1)	(2,1,2)edge(1,0,2)	(2,1,2)edge(2,0,2)  (2,1,2)edge(2,2,2);
			\end{scope}
			\begin{scope}[every path/.style={olive,arrows={>-Stealth[fill=white]}}]\path
				(2,1,3)edge(1,1,1)	(2,1,3)edge(3,1,3)	(2,1,3)edge(1,2,1)	(2,1,3)edge(3,2,3)
				(2,1,3)edge(3,0,2)	(2,1,3)edge(1,0,3)	(2,1,3)edge(2,0,3)	(2,1,3)edge(2,2,3);
			\end{scope}
			\node[align=center] at(2,-1,2) {(II) $S_{X_j}$};
		\end{tikzpicture}
		\begin{tikzpicture}[scale=.9]
			\foreach \x in {1,2,3}\foreach \y in {1,2,3}{
				\fill (\x,0,\y)circle[radius=1.5pt];
				\fill (\x,1,\y)circle[radius=1.5pt];
				\fill (\x,2,\y)circle[radius=1.5pt];}
			\begin{scope}[every path/.style={blue,arrows={>-Stealth[fill=white]}}]
				\path (1,2,2)edge(1,2,3)	(1,2,2)edge(3,2,1)	(1,2,2)edge(1,0,1)	(1,2,2)edge(2,0,3)
				(1,2,2)edge(1,1,3)	(1,2,2)edge(3,1,1)	(1,2,2)edge(1,0,2);
			\end{scope}
			\begin{scope}[every path/.style={purple,arrows={>-Stealth[fill=white]}}]\path
				(2,2,2)edge(2,2,3)	(2,2,2)edge(1,2,1)	(2,2,2)edge(2,0,1)	(2,2,2)edge(3,0,3)
				(2,2,2)edge(2,1,3)	(2,2,2)edge(1,1,1)	(2,2,2)edge(2,0,2);
			\end{scope}
			\begin{scope}[every path/.style={olive,arrows={>-Stealth[fill=white]}}]\path
				(3,2,2)edge(3,2,3)	(3,2,2)edge(2,2,1)	(3,2,2)edge(3,0,1)	(3,2,2)edge(1,0,3)
				(3,2,2)edge(3,1,3)	(3,2,2)edge(2,1,1)	(3,2,2)edge(3,0,2);
			\end{scope}
			\node[align=center] at(2,-1,2) {(III) $S_{Y_i}$};
		\end{tikzpicture}
            \caption{The construction for $n=27$.}
		\label{figure:27}
\end{figure}
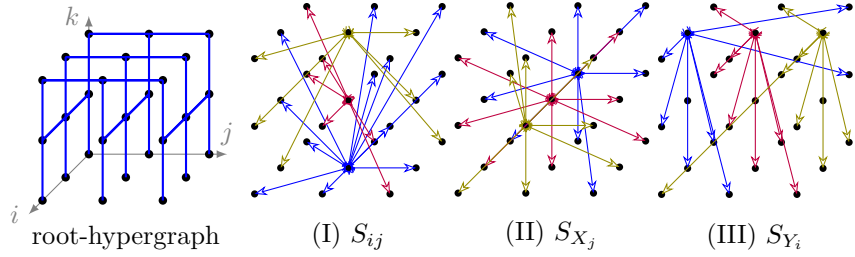

Due the symmetry arising from indices $i,j$ modulo 3, it suffices to check that for a fixed pair $(i,j)$ and for each $0\le k\le 2$, all edges of the form $(i,j,k)(i',j',k')$, where $k'\ge k$, are covered by the construction.

For the $8$ edges of the form $(i,j,0)(i',j',0)$, $4$ of them are covered by $S_{{ij}}$ while the others are covered by $S_{{i(j-1)}}$, $S_{{(i-1)j}}$, $S_{{(i-1)(j+1)}}$ and $S_{{(i+1)(j+1)}}$.

For the $9$ edges of the form $(i,j,0)(i',j',1)$, $4$ of them are covered by $S_{{ij}}$ while the others are covered by $S_{{(i+1)j}}$, $S_{{(i-1)(j-1)}}$, $S_{X_j}$, $S_{X_{(j+1)}}$ and $S_{X_{(j-1)}}$.

For the $9$ edges of the form $(i,j,0)(i',j',2)$, $4$ of them are covered by $S_{{ij}}$ while the others are covered by $S_{{i(j+1)}}$, $S_{{(i+1)(j-1)}}$, $S_{Y_i}$, $S_{Y_{(i+1)}}$ and $S_{Y_{(i-1)}}$.

For the $8$ edges of the form $(i,j,1)(i',j',1)$, $4$ of them are covered by $S_{{ij}}$ and $S_{X_j}$, while the others are covered by $S_{{(i-1)j}}$, $S_{{(i+1)(j+1)}}$, $S_{X_{(j+1)}}$ and $S_{X_{(j-1)}}$.

For the $9$ edges of the form $(i,j,1)(i',j',2)$, $5$ of them are covered by $S_{{ij}}$ and $S_{X_j}$, while the others are covered by $S_{{(i-1)(j+1)}}$, $S_{{i(j-1)}}$, $S_{Y_{(i+1)}}$ and $S_{Y_{(i-1)}}$.

For the $8$ edges of the form $(i,j,2)(i',j',2)$, $4$ of them are covered by $S_{{ij}}$ and $S_{Y_i}$, while the others are covered by $S_{{i(j-1)}}$, $S_{{(i-1)(j+1)}}$, $S_{Y_{(i+1)}}$ and $S_{Y_{(i-1)}}$.


\subsection{Blowup}\label{subsec:blowup}

\begin{lem}\label{lem:blowup}
For integers $k,n,m,t\ge 1$, if $m\le n-2$ and $F_k(n)\le m$, then $F_k(tn)\le tm$.
\end{lem}

\begin{proof}
Let $\{P_1,P_2,\dots,P_n\}$ be the vertex set of $K_n$, and $S_1,S_2,\dots,S_m$ be the k-star-forests decomposing $K_n$. Denote the vertices of $K_{tn}$ by $Q_{ab}(a\in\{1,2,\dots,n\}, b\in\{1,2,\dots,t\})$, and construct k-star-forests in $K_{tn}$ as follows:

For each $1\le j\le m$, we construct  $\tilde{S}_{j1},\tilde{S}_{j2},\dots,\tilde{S}_{jt}$ such that centers of $\tilde{S}_{jb}$ are $\{Q_{ib}~\big|~P_i~\text{is a center of}~S_j\}$, and
	
$\tilde{S}_{jb}$ contains edges 
$\begin{cases}
Q_{ib}Q_{ub'},~~ \forall b'\in\{1,2,\dots,t\},~\text{if }P_iP_u\in S_j,\\
Q_{ib}Q_{ib'},~~ \forall b'\neq b.
\end{cases}$
	
By \Cref{prop:no-iso}, the condition $m\le n-2$ guarantees that every $P_i$ acts as a center in some forest at least once. Thus for each $1\le i\le n$, all edges inside $\mathcal{Q}_i:=\{Q_{ib}| 1\le b\le t\}$ are covered by $\tilde{S}_{j1},\dots,\tilde{S}_{jt}$ if $P_i$ is a center of $S_j$. Further, since $S_j (1\le j\le m)$ cover all edges in $K_n$, edges between $\mathcal{Q}_i$ and $\mathcal{Q}_u$ are covered by $\tilde{S}_{j1},\dots,\tilde{S}_{jt}(P_iP_u\in S_j)$ for all $1\le i<u\le n$.
	
There might be some edges covered more than once by $\tilde{S}_{ab}$. In that case, one needs to delete some edges in some $\tilde{S}_{ab}$ to make the construction a decomposition. 
\end{proof}

\Cref{thm:3-star} follows by the construction for $n=27$ and \Cref{lem:blowup}.

\section{Decompostion into 4-star forests}\label{sec:4-star}

In this section, we present constructions for decompositions of $K_n$ into $\frac{n}{2}+2$ $4$-star forests when $n\equiv 4 (\mod 12)$, hence proving the Theorem~\ref{thm:4-star}.

\subsection{The construction for $n=16$}\label{subsec:16}
Let $A_0\sqcup B\sqcup C\sqcup A_1$ be the vertex set of $K_{16}$ where 
$$A_0=\{A_0(0),A_0(1),A_0(2),A_0(3)\},$$ 
$$B=\{B(0),B(1),B(2),B(3)\},$$ 
$$C=\{C(0),C(1),C(2),C(3)\}$$ and 
$$A_1=\{A_1(0),A_1(1),A_1(2),A_1(3)\}.$$
Here the indices in parentheses are modulo 4. 

\begin{figure}[h]
	\centering
	\begin{tikzpicture}
		\fill[olive!20] (-0.4,-1)rectangle(0.4,3.3) (2.6,-1)rectangle(3.4,3.3);
		\fill[orange!20] (0.6,-1)rectangle(1.4,3.3);
		\fill[violet!20] (1.6,-1)rectangle(2.4,3.3);
		\foreach \y  in {0,1,2,3}{
			\node[circle,inner sep=2,fill,label={[scale=.8]270:$A_0(\y)$}] at(0,\y) {};
			\node[circle,inner sep=2,fill,label={[scale=.8]270:$B(\y)$}] at(1,\y) {};
			\node[circle,inner sep=2,fill,label={[scale=.8]270:$C(\y)$}] at(2,\y) {};
			\node[circle,inner sep=2,fill,label={[scale=.8]270:$A_1(\y)$}] at(3,\y) {};
			\draw[rounded corners] (-0.4,\y-0.5)rectangle(3.4,\y+0.3);
			\node at(3.9,\y-0.2) {$X_{\y}$};
		}
		\node at(0,-0.8){$A_0$}; \node at(1,-0.8){$B$}; \node at(2,-0.8){$C$}; \node at(3,-0.8){$A_1$};
	\end{tikzpicture}
        \caption{Labeling and grouping the vertices of $K_{16}$.}
\end{figure}
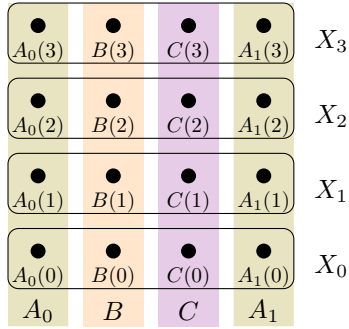
 
Let $B=\{B(0),B(1),B(2),B(3)\}$ be the centers of the 4-star-forest $S_{B}$ consisting of stars
$$
S (B(i); A_0(i), C(i+2), A_1(i) ),~~~0\le i\le 3.
$$

Let $C=\{C(0),C(1),C(2),C(3)\}$ be the centers of the 4-star-forest $S_{C}$ consisting of stars
$$
S (C(i); A_0(i-1), B(i), A_1(i) ),~~~0\le i\le 3.
$$

For $0\le i\le 3$, let 
$$X_i=\{A_0(i), B(i), C(i), A_1(i)\}$$
be the centers of the 4-star-forest $S_{X_i}$ consisting of stars
$$
S (A_0(i); A_0(i-1) ),
$$
$$
S (B(i); A_0(i+2), B(i-1), B(i+2), C(i+1), A_1(i+1) ),
$$
$$
S (C(i); A_0(i+1), B(i+1), C(i-1), C(i+2), A_1(i-1) ),
$$
and
$$
S (A_1(i); A_1(i+2) ).
$$

For $0\le i\le 1$, let $Y_i=\{A_0(2i),A_0(2i+1)\}$ be the centers of the 2-star-forest $S_{Y_i}$ consisting of stars
$$
S (A_0(j); B(j-1), B(j+1), C(j), C(j+2), A_0(j+2) ),~~~j\in\{2i,2i+1\}.
$$

For $0\le i\le 1$, let $Z_i=\{A_1(i),A_1(i+2)\}$ be the centers of the 2-star-forest $S_{Z_i}$ consisting of stars
$$
S (A_1(j); B(j+1), B(j+2), C(j-1), C(j+2), A_1(j+1) ),~~~j\in\{i,i+2\}.
$$

It is easy to check that the star-forests constructed above cover all edges of $K_{16}$ except for those between $A_0$ and $A_1$. For example, to check that all edges between $B$ and $C$ are covered, by symmetry arising from $i$ modulo 4, it suffices to check that, for fix $i$ and every $j\in \{i-1,~i,~i+1,~i+2\}$, edges of the form $B(i)C(j)$ are covered. Indeed, edges of the form $B(i)C(i-1)$ is covered by $S_{X_{i-1}}$, edges of the form $B(i)C(i)$ is covered by $S_C$, edges of the form $B(i)C(i+1)$ is covered by $S_{X_i}$, and edges of the form $B(i)C(i+2)$ is covered by $S_B$. Similarly, we can check that the edges inside each of $A_0$, $B$, $C$, and $A_1$ are covered, and that the edges between $A_0$ and $B$, $A_0$ and $C$, $A_1$ and $B$, and $A_1$ and $C$ are covered. 
We omit the details here.

We cannot find a symmetric way to cover edges between $A_0$ and $A_1$. Instead, those edges are covered by adding edges to $S_{Y_i}$ and $S_{Z_i}$ in an asymmetric way as shown in \Cref{fig:A0A1}.
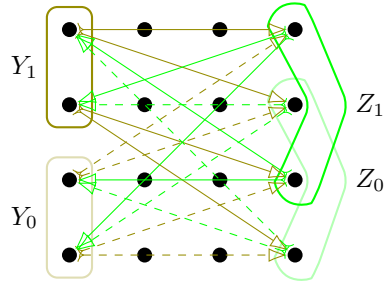
\begin{figure}[h]
		\centering
		\begin{tikzpicture}[rounded corners]
			\foreach \x in {0,1,2,3} \foreach \y in {0,1,2,3}{
				\node[circle,fill=black,inner sep=2](\x\y)at(\x,\y){};
			}
			\draw[olive,thick,opacity=0.3](-0.3,-0.3)rectangle(0.3,1.3);\node at(-0.6,0.5) {$Y_0$};
			\draw[olive,thick](-0.3,1.7)rectangle(0.3,3.3);\node at(-0.6,2.5) {$Y_1$};
			\foreach \y/\op in {0/0.3,1/1}{
				\draw[green,thick,opacity=\op](2.8,\y+0.1)arc(150:330:0.2828)--(3.7,\y+1)--(3.2,\y+2.2)arc(30:210:0.2828)--(3.2,\y+1)--cycle;
				\node at(4,\y+1) {$Z_\y$};
			}
			\draw[olive,>-{Latex[fill=none,length=7]}](03)edge(33)edge(32)
			(02)edge(31)edge(30);
			\draw[olive,>-{Latex[fill=none,length=7]},dashed](01)edge(33)edge(32)
			(00)edge(31)edge(30);
			\draw[green,>-{Latex[fill=none,length=7]}](33)edge(00)edge(02)
			(31)edge(01)edge(03);
			\draw[green,>-{Latex[fill=none,length=7]},dashed](32)edge(00)edge(02)
			(30)edge(01)edge(03);
		\end{tikzpicture}
            \caption{Covering edges between $A_0$ and $A_1$.}
            \label{fig:A0A1}
\end{figure}

Note that the number of 4-star-forests in this construction is $10=n/2+2$. This completes the proof of Theorem~\ref{thm:4-star} when $n=16$.

\subsection{Root-hypergraphs for $n=12m+4$}\label{subsec:12m+4}
In the rest of this section we extend the construction for $n=16$ to $n=12m+4$. Let us first describe the root-hypergraphs of our constructions. We partition the vertices of $K_{12m+4}$ into $3m+1$ sets $A_0,\dots, A_m$, $B_0,\dots, B_{m-1}$ and $C_0,\dots, C_{m-1}$ such that all of these sets have size four. For all $0\le i\le m$, we label the vertices in $A_i$ by $A_i(0),~A_i(1),~A_i(2)$ and $A_i(3)$, where the indices in parentheses are modulo 4. We label the vertices in $B_i$ and $C_i$ similarly. The edges in the root hypergraphs are 
$$X_{ki}=\{A_k(i), B_k(i), C_k(i), A_{k+1}(i)\}~(0\le k\le m-1, 0\le i\le 3),$$
$$B_k=\{B_k(0), B_k(1), B_k(2), B_k(3)\}~(0\le k\le m-1),$$
$$C_k=\{C_k(0), C_k(1), C_k(2), C_k(3)\}~(0\le k\le m-1),$$
$$Y_i=\{A_0(2i), A_0(2i+1)\}~(0\le i\le 1),$$
$$Z_i=\{A_m(i), A_m(i+2)\}~(0\le i\le 1).$$

 See \Cref{figure:12m+4} for an illustration of the root-hypergraph. Note that the number of edges here is $6m+4=n/2+2$. One can easily check that the root-hypergraph of our construction for $K_{16}$ is consistent with \Cref{figure:12m+4} when taking $m=1$.
\begin{figure}[h]
		\centering
		\begin{tikzpicture}
			\foreach \x/\y in {0/0,3/1,6/2,8/m-1,11/m}
				{\fill[olive!20](\x-0.4,0.2)rectangle(\x+0.4,4.4);
				\node at(\x,0.4) {$A_{\y}$};}
			\foreach \x/\y in {1/0,4/1,9/m-1}
				{\fill[orange!20](\x-0.4,0.2)rectangle(\x+0.4,4.4);
				\node at(\x,0.4) {$B_{\y}$};
				\fill[violet!20](\x+0.6,0.2)rectangle(\x+1.4,4.4);
				\node at(\x+1,0.4) {$C_{\y}$};}
			\foreach  \x in {0,3,8} \foreach \y in{1,2,3,4}
			{\draw[rounded corners=5pt,thick](\x-0.3,\y-0.3)rectangle(\x+3.3,\y+0.3);}
			\foreach \x in {1,2,4,5,9,10}
			{\draw[rounded corners=5pt,thick](\x-0.3,0.7)rectangle(\x+0.3,4.3);}
			\foreach \y in{1,2,3,4}
			{\draw[rounded corners=5pt,thick](5.7,\y-0.3)rectangle(8.3,\y+0.3);}
			\foreach \y in {1,2}
			{\draw[rounded corners=5pt,thick](-0.3,2*\y-1.3)rectangle(0.3,2*\y+0.3);
			\draw[thick,rounded corners](10.8,\y+0.1)arc(150:330:0.2828)--(11.7,\y+1)--(11.2,\y+2.2)arc(30:210:0.2828)--(11.2,\y+1)--cycle;}
			\fill[white](6.5,0.3)rectangle(7.5,4.4);
			\node foreach \x in{0,...,6,8,9,10,11} foreach \y in{1,2,3,4}[inner sep=2,fill,circle]at(\x,\y) {};
			\node foreach \y in{0.5,1,2,3,4} at(7,\y) {. . .};
			\node foreach \y in{0.7,1.7,2.7,3.7,1.3,2.3,3.3,4.3}[blue] at(7,\y) {. . . .};
		\end{tikzpicture}
            \caption{The root-hypergraph for $n=12m+4$.}
            \label{figure:12m+4}
\end{figure}
\subsection{Short-distance Rule}\label{subsec:short-dist}
For each edge $E$ in the root-hypergraph, let $S_E$ denote the star-forest using vertices in $E$ as centers. The `Short-distance rule' focuses on the edges connecting pair of vertices with `short distance', i.e. edges inside each of $A_i$, $B_i$, $C_i$ and the edges between $A_i$ and $B_i$, $A_i$ and $C_i$, $A_{i+1}$ and $B_i$, $A_{i+1}$ and $C_i$, and $B_i$ and $C_i$:

For $0\le k\le m-1$, $S_{B_k}$ contains the stars
$$
S (B_k(i); A_k(i), C_k(i+2), A_{k+1}(i) ),~(0\le i\le 3),
$$
and $S_{C_k}$ contains the stars
$$
S (C_k(i); A_k(i-1), B_k(i), A_{k+1}(i) ),~(0\le i\le 3).
$$

For $0\le i\le 3,~0\le k\le m-1$, $S_{X_{ki}}$ contain the stars
\[\begin{cases}
		S (A_k(i); A_k(i-1), B_{k-1}(i+1), B_{k-1}(i+2), C_{k-1}(i-1), C_{k-1}(i+2) ),\\
		S (B_k(i); A_k(i+2), B_k(i-1), B_k(i+2), C_k(i+1), A_{k+1}(i+1) ),\\
		S (C_k(i); A_k(i+1), B_k(i+1), C_k(i-1), C_k(i+2), A_{k+1}(i-1) ),\\
		S (A_{k+1}(i); A_{k+1}(i+2), B_{k+1}(i-1), B_{k+1}(i+1), C_{k+1}(i), C_{k+1}(i+2) ).
\end{cases}\]

And for $0\le i\le 1$, $S_{Y_i}$ contain the stars
$$
S (A_0(j); B_0(j-1), B_0(j+1), C_0(j), C_0(j+2), A_0(j+2) )~(j\in\{2i,2i+1\}),
$$
and $S_{Z_i}$ contain the stars
$$
S (A_m(j); B_{m-1}(j+1), B_{m-1}(j+2), C_{m-1}(j-1), C_{m-1}(j+2), A_m(j-1) ),
$$
($j\in\{i,i+2\}$).
	
Similar to what we have done to $K_{16}$, one can check that the construction above covers all `short-distance' edges. More formally, it covers all of the following edges:
$$
\{P_k(i)P_k(j)|P\in \{A, B, C\}, 0\le k\le m-1, 0\le i\neq j\le 3\},
$$
$$
\{P_l(j)A_k(i)|P\in \{B, C\}, 0\le l\le m-1, k\in\{l, l+1\}, 0\le i,j\le 3\},\\
$$
and
$$
\{B_k(i)C_k(j)|0\le k\le m-1, 0\le i,j\le 3\}.
$$

\subsection{Long-distance Rule}\label{subsec:long-dist}
The `Long-distance Rule' focuses on the edges connected pair of vertices not in $A_0\cup A_m$ with `Long Distance'. Precisely, for $0\le k\le m-1, 0\le i\le3$,
let $S_{X_{ki}}$ contain the following stars:
\begin{align*}
		\smashoperator[r]{\bigcup_{\substack{0\le j\le k-1 \\ 0\le j'\le k-2}}} 
        S (A_k(i); A_{j}(i+1), 
        A_{j}(i-1),
        B_{j'}(i-1),
        B_{j'}(i),
        C_{j'}(i), 
        C_{j'}(i+2) ),\\
        \bigcup_{\substack{0\le j\le k-1 \\ k+1\le l\le m-1}}
		S (B_k(i); A_{j}(i),
        B_{j}(i+1),
        C_{j}(i+1), 
        A_{l+1}(i-1), 
        B_l(i), 
        C_l(i) ),\\
        \bigcup_{\substack{0\le j\le k-1 \\ k+1\le l\le m-1}}
		S (C_k(i+2); A_{j}(i),
        B_{j}(i),
        C_{j}(i+1), 
        A_{l+1}(i-1), 
        B_l(i), 
        C_l(i) ),\\
        \smashoperator[r]{\bigcup_{l=k+2}^{m-1} }
		S (A_{k+1}(i); 
        A_l(i), 
        A_l(i+2), 
        B_l(i-1), 
        B_l(i+1), 
        C_l(i-1), 
        C_l(i+1), \\
        A_m(i), A_m(i+1) ).
\end{align*}
For $0\le k\le m-1$,
let $S_{B_k}$ contain the following stars:
\begin{align*}
        \smash{\bigcup_{\substack{0\le j\le k-1 \\ k+1\le l\le m-1 \\ k+2\le l'\le m-1 }}}
		S (B_k(i); A_j(i+2), A_{l'}(i+2), A_m(i), B_j(i+2), B_l(i+1),\\ C_j(i-1), C_l(i+1) ),\\~~0\le i\le3,
\end{align*}
and let $S_{C_k}$ contain the following stars:
\begin{align*}
        \smash{\bigcup_{\substack{0\le j\le k-1 \\ k+1\le l\le m-1 \\ k+2\le l'\le m-1 }}}
		S (C_k(i); A_j(i), A_{l'}(i-1), A_m(i+2), B_j(i+1), B_l(i), \\ C_j(i), C_l(i-1) )\\ ~~0\le i\le3.
\end{align*}
So far, we have covered all of the long-distance edges between $V(K_{12m+4})\setminus(A_0\cup A_m)$. We have also covered some of the `long-distance' edges containing vertices in $A_0\cup A_m$. It seems complicated to check this, and one can refer to subsection \ref{4sfCon} to make sense of the `Long-distance Rule' of our construction.
	
\subsection{Boundry Rule}\label{subsec:boundary}
To cover the remaining `long-distance' edges containing vertices in $A_0\cup A_m$, we add more edges into the star-forests $S_{Y_i}$'s and $S_{Z_i}$'s. For $0\le i\le 1$,
let $S_{Y_i}$ contain the following stars:
\begin{align*}
        \smash{\bigcup_{l=0}^{m-1}}
		S (A_0(2i+j); A_l(2i+j), A_l(2i+j+2), B_l(2i+j+1), B_l(2i+j-1),\\ C_l(2i+j+1), C_l(2i+j-1) ),~~0\le j\le1,
\end{align*}
and let $S_{Z_i}$ contain the following stars:
\begin{align*}
        \smash{\bigcup_{l=0}^{m-1}}
		S (A_m(i+2j); A_l(i+2j+1), A_l(i+2j+2), B_l(i+2j+2), B_l(i+2j-1),\\ C_l(i+2j), C_l(i+2j+1) ),~~0\le j\le1.
\end{align*}
Finally, we add edges between $A_0$ and $A_m$ to $S_{Y_i}$ and $S_{Z_i}$ as shown in \Cref{fig:A0Am}, actually exactly the same as what we have done to $K_{16}$.

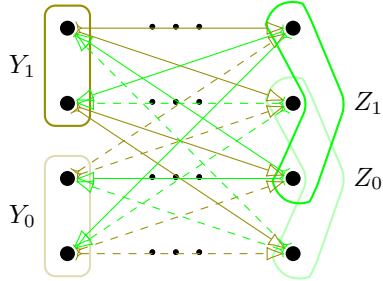
\begin{figure}[h]
		\centering
		\begin{tikzpicture}[rounded corners]
			\foreach \x in {0,3} \foreach \y in {0,1,2,3}{
				\node[circle,fill=black,inner sep=2](\x\y)at(\x,\y){};
				\node[scale=2] at(1.5,\y) {$\cdots$};
			}
			\draw[olive,thick,opacity=0.3](-0.3,-0.3)rectangle(0.3,1.3);\node at(-0.6,0.5) {$Y_0$};
			\draw[olive,thick](-0.3,1.7)rectangle(0.3,3.3);\node at(-0.6,2.5) {$Y_1$};
			\foreach \y/\op in {0/0.3,1/1}{
				\draw[green,thick,opacity=\op](2.8,\y+0.1)arc(150:330:0.2828)--(3.7,\y+1)--(3.2,\y+2.2)arc(30:210:0.2828)--(3.2,\y+1)--cycle;
				\node at(4,\y+1) {$Z_\y$};
			}
			\draw[olive,>-{Latex[fill=none,length=7]}](03)edge(33)edge(32)
			(02)edge(31)edge(30);
			\draw[olive,>-{Latex[fill=none,length=7]},dashed](01)edge(33)edge(32)
			(00)edge(31)edge(30);
			\draw[green,>-{Latex[fill=none,length=7]}](33)edge(00)edge(02)
			(31)edge(01)edge(03);
			\draw[green,>-{Latex[fill=none,length=7]},dashed](32)edge(00)edge(02)
			(30)edge(01)edge(03);
		\end{tikzpicture}
            \caption{Covering edges between $A_0\cup A_m$.}
            \label{fig:A0Am}
\end{figure}

 \subsection{Conclusion}\label{4sfCon}
 
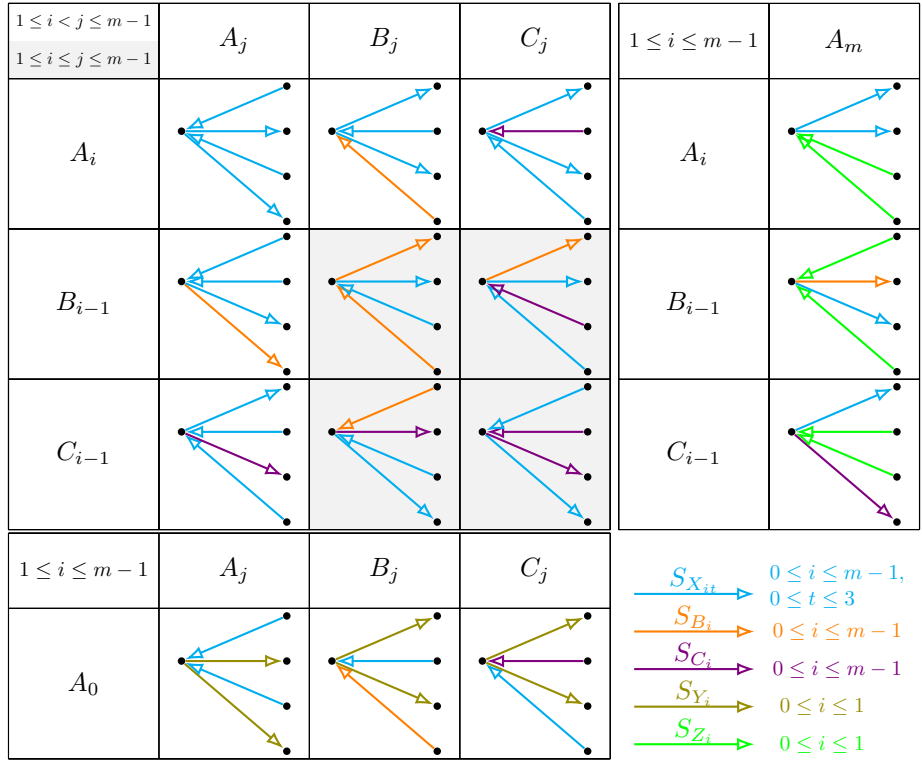
\begin{figure}[t]
		
		\begin{tikzpicture}
			\fill[gray!10] (4,0)rectangle(8,4) (0,6)rectangle(2,6.5);
			\foreach \y in {0,2,4,6,7}{\foreach \x in {0,2,4,6,8}{\draw (\x,0)edge(\x,7) (0,\y)edge(8,\y);}}
			\node[scale=.65] at (1,6.75) {$1\le i< j\le m-1$};
			\node[scale=.65] at (1,6.25) {$1\le i\le j\le m-1$};
			\node at (1,5) {$A_i$};\node at (1,3) {$B_{i-1}$};\node at (1,1) {$C_{i-1}$};
			\node at(3,6.5) {$A_j$};\node at(5,6.5) {$B_j$};\node at(7,6.5){$C_j$};
			\foreach \x in {3,5,7}{\foreach \y in {1,3,5}{
			\node[circle,fill,inner sep=1](\x\y)at(\x-0.7,\y+0.3) {};
			\node[circle,fill,inner sep=1](\x\y0)at(\x+0.7,\y-0.9) {};
			\node[circle,fill,inner sep=1](\x\y1)at(\x+0.7,\y-0.3) {};
			\node[circle,fill,inner sep=1](\x\y2)at(\x+0.7,\y+0.3) {};
			\node[circle,fill,inner sep=1](\x\y3)at(\x+0.7,\y+0.9) {};
			}}
			\draw[thick] (35)edge[-{Latex[fill=none]},cyan](350)edge[{Latex[fill=none]}-,cyan](351)edge[-{Latex[fill=none]},cyan](352)edge[{Latex[fill=none]}-,cyan](353)
			(55)edge[{Latex[fill=none]}-,orange](550)edge[-{Latex[fill=none]},cyan](551)edge[{Latex[fill=none]}-,cyan](552)edge[-{Latex[fill=none]},cyan](553)
			(75)edge[{Latex[fill=none]}-,cyan](750)edge[-{Latex[fill=none]},cyan](751)edge[{Latex[fill=none]}-,violet](752)edge[-{Latex[fill=none]},cyan](753)
			(33)edge[-{Latex[fill=none]},orange](330)edge[-{Latex[fill=none]},cyan](331)edge[{Latex[fill=none]}-,cyan](332)edge[{Latex[fill=none]}-,cyan](333)
			(53)edge[{Latex[fill=none]}-,orange](530)edge[{Latex[fill=none]}-,cyan](531)edge[-{Latex[fill=none]},cyan](532)edge[-{Latex[fill=none]},orange](533)
			(73)edge[{Latex[fill=none]}-,cyan](730)edge[{Latex[fill=none]}-,violet](731)edge[-{Latex[fill=none]},cyan](732)edge[-{Latex[fill=none]},orange](733)
			(31)edge[{Latex[fill=none]}-,cyan](310)edge[-{Latex[fill=none]},violet](311)edge[{Latex[fill=none]}-,cyan](312)edge[-{Latex[fill=none]},cyan](313)
			(51)edge[-{Latex[fill=none]},cyan](510)edge[{Latex[fill=none]}-,cyan](511)edge[-{Latex[fill=none]},violet](512)edge[{Latex[fill=none]}-,orange](513)
			(71)edge[-{Latex[fill=none]},cyan](710)edge[-{Latex[fill=none]},violet](711)edge[{Latex[fill=none]}-,violet](712)edge[{Latex[fill=none]}-,cyan](713);
		\end{tikzpicture}
		\begin{tikzpicture}
			\foreach \y in {0,2,4,6,7}{\foreach \x in {0,2,4}{\draw (\x,0)edge(\x,7) (0,\y)edge(4,\y);}}
			\node[scale=.8] at (1,6.5) {$1\le i\le m-1$};
			\node at (1,5) {$A_i$};\node at (1,3) {$B_{i-1}$};\node at (1,1) {$C_{i-1}$};
			\node at(3,6.5) {$A_m$};
			\foreach \y in {1,3,5}{
					\node[circle,fill,inner sep=1](\y)at(3-0.7,\y+0.3) {};
					\node[circle,fill,inner sep=1](\y0)at(3+0.7,\y-0.9) {};
					\node[circle,fill,inner sep=1](\y1)at(3+0.7,\y-0.3) {};
					\node[circle,fill,inner sep=1](\y2)at(3+0.7,\y+0.3) {};
					\node[circle,fill,inner sep=1](\y3)at(3+0.7,\y+0.9) {};
			}
			\draw[thick]
			(5)edge[{Latex[fill=none]}-,green](50)edge[{Latex[fill=none]}-,green](51)edge[-{Latex[fill=none]},cyan](52)edge[-{Latex[fill=none]},cyan](53)
			(3)edge[{Latex[fill=none]}-,green](30)edge[-{Latex[fill=none]},cyan](31)edge[-{Latex[fill=none]},orange](32)edge[{Latex[fill=none]}-,green](33)
			(1)edge[-{Latex[fill=none]},violet](10)edge[{Latex[fill=none]}-,green](11)edge[{Latex[fill=none]}-,green](12)edge[-{Latex[fill=none]},cyan](13);
		\end{tikzpicture}
		\begin{tikzpicture}
			\foreach \y in {0,2,3}{\foreach \x in {0,2,4,6,8}{\draw (\x,0)edge(\x,3) (0,\y)edge(8,\y);}}
			\node[scale=.8] at (1,2.5) {$1\le i\le m-1$};
			\node at (1,1) {$A_0$};
			\node at(3,2.5) {$A_j$};\node at(5,2.5) {$B_j$};\node at(7,2.5){$C_j$};
			\foreach \x in {3,5,7}{
					\node[circle,fill,inner sep=1](\x)at(\x-0.7,1+0.3) {};
					\node[circle,fill,inner sep=1](\x0)at(\x+0.7,1-0.9) {};
					\node[circle,fill,inner sep=1](\x1)at(\x+0.7,1-0.3) {};
					\node[circle,fill,inner sep=1](\x2)at(\x+0.7,1+0.3) {};
					\node[circle,fill,inner sep=1](\x3)at(\x+0.7,1+0.9) {};
			}
			\draw[thick]
			(3)edge[-{Latex[fill=none]},olive](30)edge[{Latex[fill=none]}-,cyan](31)edge[-{Latex[fill=none]},olive](32)edge[{Latex[fill=none]}-,cyan](33)
			(5)edge[{Latex[fill=none]}-,orange](50)edge[-{Latex[fill=none]},olive](51)edge[{Latex[fill=none]}-,cyan](52)edge[-{Latex[fill=none]},olive](53)
			(7)edge[{Latex[fill=none]}-,cyan](70)edge[-{Latex[fill=none]},olive](71)edge[{Latex[fill=none]}-,violet](72)edge[-{Latex[fill=none]},olive](73);
		\end{tikzpicture}
		\begin{tikzpicture}
			\clip (0,0.3)rectangle(4,3);
			\draw[thick] (.2,2.5)edge[-{Latex[fill=none]},cyan]node[above,inner sep=0,cyan]{$S_{X_{it}}$}(1.8,2.5)
			(.2,2)edge[-{Latex[fill=none]},orange]node[above,inner sep=0,orange]{$S_{B_i}$}(1.8,2)
			(.2,1.5)edge[-{Latex[fill=none]},violet]node[above,inner sep=0,violet]{$S_{C_i}$}(1.8,1.5)
			(.2,1)edge[-{Latex[fill=none]},olive]node[above,inner sep=0,olive]{$S_{Y_i}$}(1.8,1)
			(.2,.5)edge[-{Latex[fill=none]},green]node[above,inner sep=0,green]{$S_{Z_i}$}(1.8,.5);;
			\node[cyan,scale=.8,align=left] at(2.9,2.6) {$0\le i\le m-1,$\\$0\le t\le 3$};
			\node[orange,scale=.8] at (2.9,2) {$0\le i\le m-1$};
			\node[violet,scale=.8] at (2.9,1.5) {$0\le i\le m-1$};
			\node[olive,scale=.8] at (2.7,1) {$0\le i\le 1$};
			\node[green,scale=.8] at (2.7,.5) {$0\le i\le 1$};
		\end{tikzpicture}
            \caption{Distribution of `long-distance' edges to star-forests.}
            \label{fig:long-distance-distribution}
\end{figure}

 To show that our construction actually cover all the `long-distance' edges, we draw \Cref{fig:long-distance-distribution}. It shows which star-forest each type of `long-distance' edges is distributed to. For example, the left-upper corner of \Cref{fig:long-distance-distribution} shows that for each pair $i,~j$ with $1\le i< j\le m-1$, the edges between $A_i$ and $A_j$ are distributed as follows: \\
 \indent edge of the form $A_i(t)A_j(t+1)$ is covered by $S_{X_{j(t+1)}}$,\\
 \indent edge of the form $A_i(t)A_j(t)$ is covered by $S_{X_{it}}$,\\
 \indent edge of the form $A_i(t)A_j(t-1)$ is covered by $S_{X_{j(t-1)}}$, and\\
 \indent edge of the form $A_i(t)A_j(t-2)$ is covered by $S_{X_{it}}$.
	
Note that all k-star-forests in our construction contain $(n-k)$ edges (for $k\in \{2,4\}$), thus they contain
$$
(n-4)\times\frac{n}{2}+(n-2)\times4=\frac{n^2}{2}=\frac{n(n-1)}{2}+\frac{n}{2}
$$
edges in total, and in fact there are exactly $\frac{n}{2}$ edges that are covered twice in our construction, which is $P_k(i)P_k(i+2)$ for every $0\le k\le m(P=A)$ or $0\le k\le m-1(P=B,C) $ and $0\le i\le 1$. This provides another perspective to check our construction.
	
\section*{Acknowledgement}

Jiaxi Nie wishes to thank János Pach for bringing this question to his attention during his visit to the Alfréd Rényi Institute of Mathematics in January 2024.

\bibliographystyle{abbrv}
\bibliography{refs}

\end{document}